\documentclass[a4paper, 10.5pt]{article}
\usepackage{amsmath}
\usepackage{amssymb,esint}
\usepackage{amscd}
\usepackage{mathrsfs}
\usepackage{xspace}
\usepackage{fancyhdr}
\usepackage{xcolor}
\usepackage{verbatim}
\usepackage{cite}
\usepackage{appendix}
\usepackage{exscale}
\usepackage{relsize}
\usepackage{srcltx} %jump between tex and dvi, used for kile
\newtheorem{theorem}{Theorem}[section]

\newtheorem{definition}[theorem]{Definition}

\newtheorem{lemma}[theorem]{Lemma}

\newtheorem{remark}[theorem]{Remark}

\newenvironment{proof}[1][Proof]{\textbf{#1.} }{\hfill\rule{0.5em}{0.5em}}
{\catcode`\@=11\global\let\AddToReset=\@addtoreset
\AddToReset{equation}{section}

\AddToReset{theorem}{section}

\title{New Regularity Criteria for Navier-Stokes and SQG Equations in Critical Spaces}
\pagestyle{fancy}
\fancyhf{} % 清空页眉和页脚的默认设置
\begin{document}
\author{
	{\bf Yiran Xu\thanks{E-mail address: yrxu20@fudan.edu.cn, Fudan University, 220 Handan Road, Yangpu, Shanghai, 200433, China. }, Ly Kim Ha\thanks{E-mail address: lkha@hcmus.edu.vn, Address 1: University of Science VNU-HCMC, Ho Chi Minh City 700000, Vietnam, Address 2: Vietnam National University, Ho Chi Minh City, Vietnam.},~ Haina Li\thanks{E-mail address: lihaina2000@163.com,
			School of Mathematics and Statistics, Beijing Institute of Technology, Beijing 100081, China.},~Zexi Wang\thanks{E-mail address: wangzexi2000@163.com, Department of Mathematics,
			Beijing Jiaotong University, Beijing 100044, PR China.}}}
\date{}  
\maketitle
\begin{abstract}
	In this paper, we investigate some priori estimates to provide
	the critical regularity criteria for incompressible
	Navier-Stokes equations on $\mathbb{R}^3$ and  super critical surface quasi-geostrophic equations on $\mathbb{R}^2$. Concerning the Navier-Stokes equations, we demonstrate that a Leray-Hopf solution $u$ is regular if $u\in L_T^{\frac{2}{1-\alpha}} \dot{B}^{-\alpha}_{\infty,\infty}(\mathbb{R}^3)$, or $u$ in Lorentz space $ L_T^{p,r}  \dot{B}^{-1+\frac{2}{p}}_{\infty,\infty}(\mathbb{R}^3)$, with $4\leq p\leq r<\infty$. Additionally, an alternative regularity condition is expressed as $u\in L_{T}^{\frac{2}{1-\alpha}}
	\dot{B}^{-\alpha}_{\infty,\infty}(\mathbb{R}^3)+{L_T^\infty\dot{B}^{-1}_{\infty,\infty}}(\mathbb{R}^3)$($\alpha\in(0,1)$), contingent upon a smallness assumption on the norm $L_T^\infty\dot{B}^{-1}_{\infty,\infty}$. For the surface quasi-geostrophic equations, we derive that  a Leray-Hopf weak solution $\theta\in L_T^{\frac{\alpha}{\varepsilon}} \dot{C}^{1-\alpha+\epsilon}(\mathbb{R}^2)$ is smooth for any $\varepsilon$ small enough. Similar to the case of Navier-Stokes equations, we derive regularity criteria in more refined spaces, i.e. Lorentz spaces $L_T^{\frac{\alpha}{\epsilon},r}\dot{C}^{1-\alpha+\epsilon}(\mathbb{R}^2)$ and addition of two critical spaces $L_{T}^{\frac{\alpha}{\epsilon}}\dot{C}^{1-\alpha+\epsilon}(\mathbb{R}^2)+{L_T^\infty\dot{C}^{1-\alpha}(\mathbb{R}^2)}$, with smallness assumption on $L_T^\infty\dot{C}^{1-\alpha}(\mathbb{R}^2)$.
\end{abstract}
%\mykeywords % 显示关键词
\section{Introduction}
\subsection{Regularity criteria for  3D incompressible Navier-Stokes equations}
\qquad
Assume $\mathbb{R}^3$ to be replete
with a viscous incompressible fluid, whose dynamics are governed by the 3D Navier-Stokes equations:
\begin{align}\label{NS}
	\begin{cases}
		\partial_t u-\Delta u+u\cdot \nabla u+\nabla p=0,\\
		\nabla\cdot u=0,\\
		u(t,x)|_{t=0}=u_0(x),
	\end{cases}
\end{align}
where $u$ represents the velocity field, and $p$ the pressure. For the dissipation term, we have chosen a kinematic viscosity value of 1 for simplicity. 

The unresolved quest for the global
existence of smooth solutions to the 3D incompressible Navier-Stokes equations constitutes a profound challenge within the field of fluid mechanics. In the seminal work presented in reference \cite{Leray}, Leray established the global existence of weak solutions to the system (1.1) for all initial data $u_0\in L^2$ satisfying divergence-free conditions. Moreover, the subclass of weak solutions adhering to the energy inequality (2.18) is denoted as Leray-Hopf weak solutions. The question of whether smooth initial data guarantees the regularity of Leray-Hopf weak solutions remains open to the best of our knowledge and is a millennium problem. Despite this, several regularity criteria have been identified.

In light of the principle that universal physical laws should be independent of the underlying units, equation \eqref{NS} remains invariant under natural scaling transformations. If $(u, p)$ is a solution to equation \eqref{NS}, then for any $\lambda > 0$, the scaled solutions
\begin{equation*}
	u_\lambda(t, x) = \lambda u(\lambda^2t, \lambda x), \quad p_\lambda(t, x) = \lambda^2p(\lambda^2t, \lambda x)
\end{equation*}
also satisfy the equation, corresponding to rescaled initial data $u_{0,\lambda}(x) = \lambda u_{0}(\lambda x)$. Such scaling transformations deduces us to consider the scaling invariant space.

%The property of scaling invariance in the context of the Navier-Stokes equation \eqref{NS} posits that for a given solution $u$ defined on $[0, T] \times \mathbb{R}^3$ with an initial data $u_0$, the rescaled function $\lambda u(\lambda^2 t, \lambda x)$ remains a solution of \eqref{NS} on the rescaled time interval $[0, \lambda^{-2} T]$ across $\mathbb{R}^3$ with the associated rescaled initial data $\lambda u_0(\lambda x)$, which deduces us to consider the scaling invariant space.

In \cite{Fujita}, Fujita and Kato proved that   $u$  is regular if $u$ satisfies 
\begin{align*}
	u\in L^p(0,T; \dot{H}^{\frac{1}{2}+\frac{2}{p}}(\mathbb{R}^3)), \quad \text{with}\quad p\in(2,\infty).
\end{align*}
And the well-known Ladyzhenskaya-Prodi-Serrin criteria \cite{Serrin,Giga,Sohr} guarantee the regularity of $u$ with condition:
\begin{align}\label{P-S}
	u\in L^p(0,T;  L^q(\mathbb{R}^3)), \quad \quad \forall p\in[2,\infty) \quad\text{with}\quad \frac{3}{q}+\frac{2}{p}=1.
\end{align}
In the investigation of Leray-Hopf weak solution regularity, it is crucial to highlight the exclusion of the endpoint case $q = 3$ and $p = \infty$ in \eqref{P-S}. Early advancements by von Wahl \cite{41Wahl} identified the space $C(0, T;L^3)$ as a notable regularity class. Further exploration in \cite{28 Kozono} revealed the regularization properties of $BV(0, T;L^3)$. Subsequent studies \cite{28 Kozono,15Da Veiga} expanded on these findings within $L^\infty(0, T;L^3)$. Kozono's \cite{29 Kozono} contributions replaced $L^3$ with a larger Lorentz space, prompting a broader exploration of such spaces \cite{8Borchers,30Kozono}. Sohr \cite{38 Sohr} extended Serrin's results by introducing Lorentz spaces in both time and space. 

J. Y. Chemin and F. Planchon generalized this to Besov space with negative indices (refer to \cite{Chemin}), that is 
\begin{align*}
	u\in L^\infty(0,T;  \dot{B}^{-1+\frac{3}{p}}_{p,\infty}(\mathbb{R}^3)), \quad\text{with}\quad 3<p<6.
\end{align*}

It is also deserved to be mentioned that  many significant regularity criteria were established in terms of only partial components of the velocity field for Navier-Stokes equation and Euler equation \cite{NewPC1,NewPC2}.  For instance, J. Y. Chemin and P. Zhang \cite{3Chemin,4Chemin}
and B. Han Z. Lei. et al. \cite{Lei} proved the regularity of $u$, if
\begin{align*}
	u^3\in L^p(0,T; \dot{H}^{\frac{1}{2}+\frac{2}{p}}(\mathbb{R}^3)), \quad \text{with}\quad p\in[2,\infty).
\end{align*}
And  in \cite{Ting}, H. Chen
, D. Fang and T. Zhang obtained that
\begin{align*}
	\partial_3u\in L^p(0,T;  L^q(\mathbb{R}^3)), \quad \quad \forall q\in(3/2,6] \quad\text{with}\quad \frac{2}{p}+\frac{3}{q}=2.
\end{align*}

In terms of the chain of critical spaces ($r\geq p$)
\begin{align*}
	L^p(0,T; \dot{H}^{\frac{1}{2}+\frac{2}{p}}(\mathbb{R}^3))\hookrightarrow  L^{p}(0,T;  L^{3+\frac{6}{p-2}}(\mathbb{R}^3))\hookrightarrow  L^p(0,T;  \dot{B}^{-1+\frac{2}{p}}_{\infty,\infty}(\mathbb{R}^3))\hookrightarrow  L^{p,r}(0,T;  \dot{B}^{-1+\frac{2}{p}}_{\infty,\infty}(\mathbb{R}^3)),
\end{align*}
we naturally  consider the regularity of $u$ with condition $u\in L^{\frac{2}{1-\alpha}}(0,T;  \dot{B}^{-\alpha}_{\infty,\infty}(\mathbb{R}^3)) $ or $u\in L^{\frac{2}{1-\alpha},r}(0,T;  \dot{B}^{-\alpha}_{\infty,\infty}(\mathbb{R}^3)) $, where $\alpha\in[0,1]$ and  $r\geq \frac{2}{1-\alpha}.$ In this paper, we develop some prior estimates to show the regularity  criteria of Navier-Stokes equations \eqref{NS}. Now, we state
our main theorems. \\
To simplify our results, we introduce following notations:
\begin{align*}
	L_T^pX=L^p(0,T;X).
\end{align*}
\begin{theorem}\label{LpB}
	Let $u$ be the unique solution of the Navier–Stokes equations \eqref{NS} with initial data $u_0 \in H^1(\mathbb{R}^3)$ and $\nabla\cdot u_0 = 0$. The solution $u$ is regular in $(0, T]$, provided that
	\begin{align}\label{mm}
		u\in L_T^{\frac{2}{1-\alpha}} \dot{B}^{-\alpha}_{\infty,\infty}(\mathbb{R}^3),\quad\alpha\in[0,1).
	\end{align}
	More specifically, we have
	\begin{align}\label{Lt}
		\|\nabla u\|_{L_T^\infty L^2(\mathbb{R}^3)}+\|\nabla^2 u\|_{L_T^2 L^2(\mathbb{R}^3)}\leq \|\nabla u_0\|_{L^2(\mathbb{R}^3)}\exp\left(C\int_{0}^{T}\|u(s)\|_{\dot{B}^{-\alpha}_{\infty,\infty}(\mathbb{R}^3)}^{\frac{2}{1-\alpha}}ds\right).
	\end{align}
\end{theorem}
\begin{remark}
	Theorem \ref{LpB} implies that a Leray-Hopf weak solution \(u\) satisfying \eqref{mm} is regular in \((0, T]\). Specifically, if \(u \in L^\infty(0, T; L^2(\mathbb{R}^3)) \cap L^2(0, T; \dot{H}^1(\mathbb{R}^3))\), then for all \(s \in (0, T)\), there exists \(t_0 \in [0, s]\) such that \(u(t_0, \cdot) \in H^1\). By the application of Theorem \ref{LpB}, it follows that \(u\) is regular in \([t_0, T]\). Since \(s\) was arbitrarily chosen, the conclusion holds for the entire interval \((0, T]\).
\end{remark}
\begin{theorem}\label{LprB}
	Let $\alpha\in(\frac{1}{2},1)$ and $u$ be the Leray-Hopf  solution of the Navier–Stokes equations \eqref{NS} with initial data $u_0 \in H^1(\mathbb{R}^3)$ and $\nabla\cdot u_0 = 0$, then for any $\frac{2}{1-\alpha}<r<\infty$,
	\begin{align*}
		\|\nabla u\|_{L^\infty_TL^2(\mathbb{R}^3)}+\|\nabla^2 u\|_{L^2_TL^2 (\mathbb{R}^3)}
		&\lesssim 
		\|\nabla u_0\|_{L^2(\mathbb{R}^3)}
		\left(1+\|\nabla u_0\|^{\frac{3-2\alpha}{2\alpha-1}}_{L^2(\mathbb{R}^3)} \|u_0\|_{L^2(\mathbb{R}^3)} \right)\exp\left(C\|u_0\|^{\frac{4}{3-2\alpha}}_{L^2(\mathbb{R}^3)}\right)\cdot\\
		&\qquad\cdot\exp\left(C\|u\|_{L_T^{\frac{2}{1-\alpha},r}\dot{B}^{-\alpha}_{\infty,\infty}(\mathbb{R}^3)}^{r}\right),
	\end{align*}
	where $C$ is a constant depending on  $\alpha$, $r$.
\end{theorem}
Moreover, for the  borderline case $L_T^\infty\dot{B}^{-1}_{\infty,\infty}(\mathbb{R}^3)$, we consider the sum space $L_{T}^{\frac{2}{1-\alpha}}
\dot{B}^{-\alpha}_{\infty,\infty}(\mathbb{R}^3)+{L_T^\infty\dot{B}^{-1}_{\infty,\infty}}(\mathbb{R}^3)$  and  have the following result:
\begin{theorem}\label{endpoint case}
	Let $\alpha\in[0,1)$ and $u$ be the Leray-Hopf  solution of the Navier–Stokes equations \eqref{NS} with initial data $u_0 \in H^1(\mathbb{R}^3)$ and $\nabla\cdot u_0 = 0$ satisfying the following structure:
	$$u(t,x)=u_1(t,x)+u_2(t,x),$$
	where, \begin{equation}
		u_1\in L_{T}^{\frac{2}{1-\alpha}}
		\dot{B}^{-\alpha}_{\infty,\infty}(\mathbb{R}^3),\quad
		u_2\in{L_T^\infty\dot{B}^{-1}_{\infty,\infty}}(\mathbb{R}^3) . 
	\end{equation}
	Then there exists $\epsilon_0>0$ such that for any $	\|u_2\|_{L_T^\infty\dot{B}^{-1}_{\infty,\infty}(\mathbb{R}^3)} \leq \epsilon\leq\epsilon_0$, we have
	\begin{align*}
		\|\nabla u\|_{L_T^\infty L^2(\mathbb{R}^3)}+\|\nabla^2 u\|_{L^2_TL^2 (\mathbb{R}^3)}
		\leq \|\nabla u_0\|_{L^2(\mathbb{R}^3)}\exp\left(C\int_{0}^{T}\|u_1(s)\|_{\dot{B}^{-\alpha}_{\infty,\infty}(\mathbb{R}^3)}^{\frac{2}{1-\alpha}}ds\right).
	\end{align*}
\end{theorem}
Now, we extend these results to the 2D supercritical surface quasi-geostrophic equations.\\
\subsection{Regularity criterion for  2D supercritical Surface Quasi-Geostrophic equations}
We consider the surface quasi-geostrophic  (SQG) equations on $\mathbb{R}^2$:
\begin{align}\label{SQG}
	\begin{cases}
		\partial_t \theta+\nabla^\perp(-\Delta)^{-\frac{1}{2}} \theta\cdot\nabla \theta+(-\Delta)^{\frac{\alpha}{2}} \theta =0,\\
		\theta(t,x)|_{t=0}=\theta_0(x),
	\end{cases}
\end{align}
where $\nabla^\perp=(-\partial_2,\partial_1)$, $\alpha\in(0,2]$ and $\theta=\theta(t,x)$ is a scalar function of $x\in\mathbb{R}^2$ and $t>0$, representing the potential temperature.

SQG equation is a vital model for large-scale fluid dynamics in Earth's atmosphere and oceans. It offers insights into material transport processes, aiding the understanding of atmospheric and oceanic circulation, as well as eddy dynamics. SQG has applications in environmental science for modeling substance dispersion and is valuable in engineering, particularly for coastal and marine systems design. 

The property of scaling invariance in the context of the SQG equations \eqref{SQG} posits that
\begin{equation}
	\theta_\lambda(t,x) = \lambda^{\alpha - 1}\theta(\lambda x, \lambda^\alpha t),
\end{equation}
where $\lambda > 0$ serves as a scaling parameter. This scaling transformation preserves the critical characteristics of a series  spaces and plays a pivotal role in the analysis of the system dynamics.

Equations \eqref{SQG} are commonly classified as supercritical, critical, and subcritical SQG for $0 < \alpha < 1$, $\alpha = 1$, and $1 < \alpha \leq 2$, respectively.  Resnick \cite{27 Resnick} rigorously established the existence of a global weak solution.

In the subcritical case ($1 < \alpha \leq 2$), scenario, Constantin and Wu \cite{10Constantin} demonstrated the unique global smooth solution emanating from sufficiently smooth initial data. 

The critical case ($\alpha = 1$) is resolved independently by Kieslev, Nazarov, and Volberg \cite{23  Kiselev}, Caffarelli and Vasseur \cite{2Caffarelli}, Kieslev and Nazarov \cite{22 Kiselev}, and Constantin and Vicol \cite{11  Constantin}, utilizing diverse sophisticated methods (see  		\cite{ci,sqgb,sqgzapde,hungperterm} and references therein).

It is still an open problem for global regularity for the supercritical case ($0 < \alpha <1$). In the
spirit of the conditional regularity results available for the 3D Navier–Stokes equations, a series of  conditional regularity results was proposed.We recall that the first conditional regularity result for solutions to (1.1) was obtained by Constantin, Majda and
Tabak \cite{NewPC5}. Chae \cite{3 Chae} established a  Ladyzhenskaya-Prodi-Serrin type regularity criterion, that is 
\begin{align*}
	\theta\in L^p(0,T;  L^q(\mathbb{R}^2)), \quad \quad \forall p\in\left(\frac{2}{\alpha},\infty\right) \quad\text{with}\quad \frac{2}{q}+\frac{\alpha}{p}\leq \alpha.
\end{align*}
Later in \cite{NewPC4,14 Constantin, NewPC6}, P. Constantin and J. Wu demonstrated that a weak solution can transform into a classical solution if 
\begin{align*}
	\theta\in L^\infty(0,T;  C^{\delta}(\mathbb{R}^2)), \quad\text{with}\quad \delta >1- \alpha.
\end{align*} 
Furthermore, the outcomes of \cite{19 HJ Dong} enhanced  criterion through the utilization of Besov spaces, the authors derived regularity conditions  in scaling invariant space as
\begin{align*}
	\theta\in L^p(0,T;  {B}^{\beta}_{q,\infty}(\mathbb{R}^2)),\quad \forall p\in[1,\infty),\quad\forall q\in(2,\infty) \quad\text{with}\quad \beta=1-\alpha+\frac{2}{q}+\frac{\alpha}{p}.
\end{align*}
Later, in \cite{Dong} and \cite{ZhuanYe} authors also obtained the regularity criterion result.

In our paper, we deal with the endpoint case of \cite{19 HJ Dong}, i.e. the case $q=\infty$, and improve the result to the Lorentz space. We obtain the following prior estimate:
\begin{theorem}\label{SQGTH}
	For any $\alpha\in(0,1), \varepsilon\in [0,\frac{\alpha}{2})$, assume $\theta$ is a viscosity  solution to \eqref{SQG} with initial data $\theta_0\in H^2(\mathbb{R}^2)$,  The solution $\theta$ is regular in $(0, T]$, provided that
	\begin{align*}
		\theta\in L_{T}^{\frac{\alpha}{\epsilon}}\dot{C}^{1-\alpha+\epsilon}(\mathbb{R}^2).
	\end{align*}
	Moreover, we have the following prior estimate: 
	\begin{align}\label{ll}
		\| \theta\|_{L^\infty_T\dot{H}^2(\mathbb{R}^2)}+\|\theta\|_{L^2_T \dot{H}^{2+\frac{\alpha}{2}}(\mathbb{R}^2)}\leq \| \theta_0\|_{\dot{H}^2(\mathbb{R}^2)}\exp\left(C\int_{0}^{T}\|\theta(s)\|_{\dot{C}^{1-\alpha+\epsilon}(\mathbb{R}^2)}^{\frac{\alpha}{\epsilon}}ds\right).
	\end{align}	
\end{theorem}
\begin{remark}
	In fact the statement holds true for Leray-Hopf weak solution, given their satisfaction of truncated energy estimates \cite{2Caffarelli}.
\end{remark}
To extend  Theorem \ref{SQGTH} to Lorentz space, the key ingredient is Lemma  \ref{Ltqr}, and we obtain the following theorem.
\begin{theorem}\label{SQGTH 1}
	Let $\alpha\in(0,1), \epsilon \in (0,\frac{\alpha^2}{8})$ and $\theta$ be a Leray-Hopf  weak solution to \eqref{SQG} with initial data $\theta_0\in H^2(\mathbb{R}^2)$, then for any $\frac{\alpha}{\epsilon}<r<\infty$,
	\begin{align*}
		\| \theta\|_{L^\infty_T\dot{H}^2(\mathbb{R}^2)}+\|\theta\|_{L^2_T \dot{H}^{2+\frac{\alpha}{2}}(\mathbb{R}^2)}
		&\lesssim 
		\|\theta_0\|_{\dot{H}^2(\mathbb{R}^2)} 
		\left(1+ \|\theta_0\|_{L^2(\mathbb{R}^2)}
		\left(1+\|\theta_0\|_{L^2(\mathbb{R}^2)}^{-\frac{8\epsilon}{\alpha^2}}\right)
		\left(1+\|\theta_0\|_{\dot{H^2}(\mathbb{R}^2)}^{\frac{4\alpha}{\alpha^2-8\epsilon}}\right)
		\right)\cdot\\
		&\qquad\cdot\exp\left(C\|\theta\|_{L_T^{\frac{\alpha}{\epsilon},r}\dot{C}^{1-\alpha+\epsilon}(\mathbb{R}^2)}^{r}\right).		
	\end{align*}
\end{theorem}
In the borderline case, parallel to the Navier-Stokes equations, we consider the sum space $L_{T}^{\frac{\alpha}{\epsilon}}\dot{C}^{1-\alpha+\epsilon}(\mathbb{R}^2) + L_T^\infty\dot{C}^{1-\alpha}(\mathbb{R}^2)$, and have the following  theorem:

\begin{theorem}\label{SQGTH 2}
	Let $\alpha\in[0,1)$ and $\theta$ be a Lerat-Hopf weak solution to \eqref{SQG} with initial data $\theta_0\in H^2(\mathbb{R}^2)$, which satisfying the following structure:
	$$\theta(t,x)=\theta_1(t,x)+\theta_2(t,x),$$
	where, \begin{equation}
		\theta_1\in L_{T}^{\frac{\alpha}{\epsilon}}
		\dot{C}^{1-\alpha+\epsilon}(\mathbb{R}^2),\quad
		\theta_2\in{L_T^\infty\dot{C}^{1-\alpha}}(\mathbb{R}^2) . 
	\end{equation}
	Then there exists $\delta_0>0$ such that for any $	\|\theta_2\|_{L_T^\infty\dot{C}^{1-\alpha}(\mathbb{R}^2)}\leq \delta \leq \delta_0$, we have
	\begin{align*}
		\|\theta\|_{L^\infty_T\dot{H}^2(\mathbb{R}^2)}+\|\theta\|_{L_T^2\dot{H}^{2+\frac{\alpha}{2}}(\mathbb{R}^2) }
		&\leq \|\theta_0\|_{\dot{H}^2(\mathbb{R}^2)}\exp\left(C\int_{0}^{T}\|\theta_1(s)\|_{C^{1-\alpha+\epsilon}(\mathbb{R}^2)}^{\frac{\alpha}{\epsilon}}ds\right).
	\end{align*}
\end{theorem}
\section{Preliminaries}

\begin{definition}
	Define Lorentz spaces: for $q>0, s>0$,
	\begin{align*}
		\|g\|_{L^{q,s}(\mathbb{R}^d)}:=
		\begin{cases}
			\left(
			q\int_{0}^{\infty}
			(\lambda^q |\{|g|>\lambda\}|)^{\frac{s}{q}} \frac{d\lambda}{\lambda}
			\right)^{\frac{1}{s}},\quad s<\infty,\\
			\sup_{\lambda>0} \lambda (|\{|g|>\lambda\}|)^{\frac{1}{q}}, \quad s=\infty.
		\end{cases}
	\end{align*}
\end{definition}
\begin{definition}[Dyadic partition of Unity and Besov spaces]
	Let $\varphi\in C_c^\infty(\{\xi\in\mathbb{R}^3:\frac{3}{4}\leq|\xi|\leq\frac{8}{3}\})$ be the cut-off function satisfying $\sum_{j\in\mathbb{Z}}\varphi(2^{-j}\xi)=1$, for all $\xi\in\mathbb{R}^3\setminus \{0\}$, and the homogeneous dyadic blocks $\dot{\Delta}_j$ is defined by $\dot{\Delta}_ju=\varphi(2^{-j}D)u$.
	Let $s\in \mathbb{R}$ and $p,r \in [1,\infty]$. Define 
	\begin{align*}
		\|u\|_{\dot{B}_{p,r}^s(\mathbb{R}^d)}:=	\begin{cases}
			\left(
			\sum_{j\in \mathbb{Z}} 2^{rjs} \|\dot{\Delta}_ju\|_{L^p}^r
			\right)^{\frac{1}{r}},\quad r<\infty,\\
			\sup_{j\in \mathbb{Z}} 2^{js} \|\dot{\Delta}_ju\|_{L^p}, \quad r=\infty.
		\end{cases}
	\end{align*}
\end{definition}
\begin{definition}[Bony's Paraproduct decomposition]\label{Bony}
	Paradifferential calculus is a mathematical tool for splitting the sum $fg=\sum_{j,k}\dot{\Delta}_jf\dot{\Delta}_kg$ into three parts:
	\begin{align*}
		&R(f,g)=T_{hh}(f,g)=\sum_{|j-k|\leq2}\dot{\Delta}_jf\dot{\Delta}_kg,\\
		&T_{hl}(f,g)=\sum_{j\geq3+k}\dot{\Delta}_jf\dot{\Delta}_kg,\\
		&T_{lh}(f,g)=\sum_{j\leq-3+k}\dot{\Delta}_jf\dot{\Delta}_kg.
	\end{align*}
	Thus, we have
	\begin{align*}
		fg=\sum_{j,k}\dot{\Delta}_jf\dot{\Delta}_kg=T_{hl}(f,g)+T_{lh}(f,g)+R(f,g).
	\end{align*}
\end{definition}
From \cite{99Bahouri}, we can obtain the following classical Lemmas.
\begin{lemma}\label{T_hl}
	Let $p,r,r_1,r_2 \in[1,\infty]$ and $s\in \mathbb{R}$. Then for any $s_0>0$, we have
	\begin{align*}
		\|T_{hl}(f,g)\|_{\dot{B}_{p,r}^{s-s_0}} \lesssim \|f\|_{\dot{B}_{p,r_1}^s} \|g\|_{\dot{B}_{\infty,r_2}^{-s_0}},
		\quad \frac{1}{r}=\min\{1,\frac{1}{r_1}+\frac{1}{r_2}\}
	\end{align*}
\end{lemma}
\begin{lemma}
	Let $p_1,p_2,r_1,r_2 \in [1,\infty]$ and $s_1,s_2 \in \mathbb{R}$. Define $\frac{1}{p}=\frac{1}{p_1}+\frac{1}{p_2}\leq 1, \frac{1}{r}=\frac{1}{r_1}+\frac{1}{r_2}\leq 1$. Then, if $s_1+s_2>0$, we have
	\begin{align*}
		\|R(f,g)\|_{\dot{B}_{p,r}^{s_1+s_2}}
		\lesssim  \|f\|_{\dot{B}_{p_1,r_1}^{s_1}} \|g\|_{\dot{B}_{p_2,r_2}^{s_2}}.
	\end{align*}
\end{lemma}
\begin{lemma}
	For any $s>0$ and $p,r \in [1,\infty]$, we have
	\begin{align*}
		\|fg\|_{\dot{B}_{p,r}^s} \lesssim \|f\|_{L^\infty} \|g\|_{\dot{B}_{p,r}^s} + \|g\|_{L^\infty} \|f\|_{\dot{B}_{p,r}^s}.
	\end{align*}
\end{lemma}
Now we start to prove some interpolation inequalities which will be used several times in the proof of the main theorems.
\begin{lemma}\label{1}
	For any $\alpha\in [0,1]$, it holds that
	\begin{align*}
		\|\nabla f\|_{L^3(\mathbb{R}^3)}\lesssim\|f\|_{\dot{B}^{-\alpha}_{\infty,\infty}(\mathbb{R}^3)}^{\frac{1}{3}}\|\nabla f\|_{L^2(\mathbb{R}^3)}^{\frac{1-\alpha}{3}}\|\nabla^2f\|^{\frac{1+\alpha}{3}}_{L^2(\mathbb{R}^3)}.
	\end{align*}
\end{lemma}
\begin{proof}
	Using interpolation inequality we have
	\begin{align*}
		\|\nabla f\|_{L^3(\mathbb{R}^3)}\lesssim \|\nabla f\|_{\dot{B}^{-1-\alpha}_{\infty,\infty}(\mathbb{R}^3)}^{\frac{1}{3}}\|\nabla f\|_{\dot{B}^{\frac{1+\alpha}{2}}_{2,2}(\mathbb{R}^3)}^{\frac{2}{3}}.
	\end{align*}
	Note that $\|\nabla f\|_{\dot{B}^{-1-\alpha}_{\infty,\infty}(\mathbb{R}^3)}\lesssim \|f\|_{\dot{B}^{-\alpha}_{\infty,\infty}(\mathbb{R}^3)}$ and $\|f\|_{\dot{B}^{s}_{2,2}(\mathbb{R}^3)}\sim\|f\|_{\dot{H}^s(\mathbb{R}^3)}$, thus
	\begin{align*}
		\|\nabla f\|_{L^3(\mathbb{R}^3)}\lesssim \| f\|_{\dot{B}^{-\alpha}_{\infty,\infty}(\mathbb{R}^3)}^{\frac{1}{3}}\|\nabla f\|_{\dot{H}^{\frac{1+\alpha}{2}}(\mathbb{R}^3)}^{\frac{2}{3}}\lesssim \| f\|_{\dot{B}^{-\alpha}_{\infty,\infty}(\mathbb{R}^3)}^{\frac{1}{3}}\|\nabla f\|_{L^2(\mathbb{R}^3)}^{\frac{1-\alpha}{3}}\|\nabla^2f\|^{\frac{1+\alpha}{3}}_{L^2(\mathbb{R}^3)}.
	\end{align*}
\end{proof}
\begin{lemma}\label{Ltqr}
	For $1\leq p\leq q\leq r\leq \infty$ and $d\in\mathbb{N}$, it holds that
	\begin{align}\label{Ltz}
		\|g\|_{L^q(\mathbb{R}^d)}^q\leq \|g\|_{L^p(\mathbb{R}^d)}^p+\|g\|_{L^{q,r}(\mathbb{R}^d)}^q\left(\log(1+ \|g\|_{L^{\infty}(\mathbb{R}^d)})\right)^{1-\frac{q}{r}},
		\\\label{Ltz2}
		\|g\|_{L^q(\mathbb{R}^d)}^q \leq 
		1+\log \left[1+\|g\|_{L^\infty(\mathbb{R}^d)} \|g\|_{L^p(\mathbb{R}^d)}^{\frac{p}{q-p}}
		\right]^{1-\frac{q}{r}} \|g\|_{L^{q,r}(\mathbb{R}^d)}^q.
	\end{align}
\end{lemma}
\begin{proof}
	To establish \eqref{Ltz} and \eqref{Ltz2}, we start by decomposing $\|g\|_{L^q(\mathbb{R}^d)}^q$ into three integrals:
	\begin{align*}
		\|g\|_{L^q(\mathbb{R}^d)}^q&\leq \int_{0}^{\lambda_2}\lambda^q|\{|g|>\lambda\}|\frac{d\lambda}{\lambda}
		+\int_{\lambda_2}^{\lambda_1}\lambda^q|\{|g|>\lambda\}|\frac{d\lambda}{\lambda}
		+\int_{\lambda_1}^{\infty}\lambda^q|\{|g|>\lambda\}|\frac{d\lambda}{\lambda}\\
		&\leq \lambda_2^{q-p} \int_{0}^{\lambda_2}\lambda^p|\{|g|>\lambda\}|\frac{d\lambda}{\lambda}
		+\left(\int_{\lambda_2}^{\lambda_1}\frac{d\lambda}{\lambda}\right)^{1-\frac{q}{r}}\left(\int_{\lambda_2}^{\lambda_1}\lambda^r|\{|g|>\lambda\}|^{\frac{r}{q}}\frac{d\lambda}{\lambda}\right)^{\frac{q}{r}}
		+\int_{\lambda_1}^{\infty}\lambda^q|\{|g|>\lambda\}|\frac{d\lambda}{\lambda}\\
		&\leq \lambda_2^{q-p} \|g\|_{L^p(\mathbb{R}^d)}^p
		+\left[\log\lambda_1-\log\lambda_2\right]^{1-\frac{q}{r}}\|g\|_{L^{q,r}(\mathbb{R}^d)}^q
		+\int_{\lambda_1}^{\infty}\lambda^q\left|\{|g|>\lambda\}\right|\frac{d\lambda}{\lambda}.
	\end{align*}
	Taking $\lambda_1=1+\|g\|_{L^\infty(\mathbb{R}^d)},\lambda_2=1$, then we obtain \eqref{Ltz}
	\begin{align*}
		\|g\|^q_{L^q(\mathbb{R}^d)}\leq \|g\|_{L^p(\mathbb{R}^d)}^{p}+\|g\|^q_{L^{q,r}(\mathbb{R}^d)}\left(\log(1+ \|g\|_{L^{\infty}(\mathbb{R}^d)})\right)^{1-\frac{q}{r}}.
	\end{align*}
	Taking $\lambda_1=\|g\|_{L^\infty(\mathbb{R}^d)}+\|g\|_{L^p(\mathbb{R}^d)}^{-\frac{p}{q-p}},\lambda_2=\|g\|_{L^p(\mathbb{R}^d)}^{-\frac{p}{q-p}}$, we get the estimate in \eqref{Ltz2}
	\begin{align*}
		\|g\|_{L^q(\mathbb{R}^d)}^q \leq 
		1+\log \left[1+\|g\|_{L^\infty(\mathbb{R}^d)} \|g\|_{L^p(\mathbb{R}^d)}^{\frac{p}{q-p}}
		\right]^{1-\frac{q}{r}} \|g\|_{L^{q,r}(\mathbb{R}^d)}^q,
	\end{align*}
	which completes the proof of \eqref{Ltz2}.
\end{proof}\\
The inequalities \eqref{Ltz} and \eqref{Ltz2} provide crucial estimates for the behavior when we give the prior estimate of Lorentz spaces.

\section{Proof of Theorem \ref{LpB}, \ref{LprB} and \ref{endpoint case}}

%%%%%%%%%%%%%%%%%%%%%%%%%%%%%%%%%%%%%%%%%%%%%%%%%%%%%%%%%%%%%%%%%%%%%%%%%%%%%%%%%%%%%%%%%%%%%%%%%%%%%%%%%%%%%%%%%%%%%%%%5

\begin{proof}[Proof of Theorem  \ref{LpB}]
	Multiply $\Delta u$ on \eqref{NS} and integral over $dx$ we have
	\begin{align*}
		\frac{1}{2}\partial_t\|\nabla u(t)\|^2_{L^2(\mathbb{R}^3)}+\|\Delta u(t)\|^2_{L^2(\mathbb{R}^3)}&=\int u\cdot\nabla u \Delta u dx=\sum_{i,j,k}\int \partial_{kk}u^i u^j\partial_ju^i dx\\
		&=-\sum_{i,j,k}\int \partial_{jk}u^i u^j\partial_ku^i dx-\sum_{i,j,k}\int \partial_{j}u^i \partial_ku^j\partial_ku^i dx\\
		&=-\sum_{i,j,k}\int \partial_{j}u^i \partial_ku^j\partial_ku^i dx,
	\end{align*}
	
	where in the last line we use $\nabla\cdot u=0$  to obtain the fact that
	\begin{align*}
		\sum_{i,j,k}\int \partial_{jk}u^i u^j\partial_ku^i dx=-\sum_{i,j,k}\int \partial_{jk}u^i u^j\partial_ku^i dx.
	\end{align*}
	Thus we have 
	\begin{align}\label{NS ineq}
		\partial_t\|\nabla u(t)\|^2_{L^2(\mathbb{R}^3)}+\|\Delta u(t)\|^2_{L^2(\mathbb{R}^3)}\leq C \|\nabla u\|_{L^3(\mathbb{R}^3)}^3.
	\end{align}
	Combining this with Lemma \ref{1} and $\|\nabla^2 u\|_{L^2(\mathbb{R}^3)}\leq C\|\Delta u\|_{L^2(\mathbb{R}^3)}$, we obtain that
	\begin{align}\label{NS ineq2}
		\partial_t\|\nabla u(t)\|^2_{L^2(\mathbb{R}^3)}+\|\nabla^2 u(t)\|^2_{L^2(\mathbb{R}^3)}
		&\leq C \|u\|_{\dot{B}^{-\alpha}_{\infty,\infty}(\mathbb{R}^3)}\|\nabla u\|_{L^2}^{1-\alpha}\|\nabla^2u\|^{1+\alpha}_{L^2(\mathbb{R}^3)}\\\nonumber
		&\leq \frac{1}{2}\|\nabla^2 u\|_{L^2}^2+C\|u\|_{\dot{B}^{-\alpha}_{\infty,\infty}(\mathbb{R}^3)}^{\frac{2}{1-\alpha}}\|\nabla u\|^{2}_{L^2(\mathbb{R}^3)}.
	\end{align}
	Absorbing $\frac{1}{2}\|\nabla^2 u\|_{L^2}^2$ by the left hand side and using Gr\"{o}nwall inequality, we obtain that 
	\begin{align*}
		\|\nabla u\|_{L_T^\infty L^2(\mathbb{R}^3)}+\|\nabla^2 u\|_{L_T^2 L^2(\mathbb{R}^3)}\leq \|\nabla u_0\|_{L^2(\mathbb{R}^3)}\exp\left(C\int_{0}^{T}\|u(s)\|_{\dot{B}^{-\alpha}_{\infty,\infty}(\mathbb{R}^3)}^{\frac{2}{1-\alpha}}ds\right),
	\end{align*}
	where $C$ is a constant depending on  $\alpha$.
\end{proof}\\
%%%%%%%%%%%%%%%%%%%%%%%%%%%%%%%%%%%%%%%%%%%%%%%%%%%%%%%%%%%%%%%%%%%%%%%%%%%%%%%%%%%%%%%%%%%%%%

\begin{proof}[Proof of Theorem  \ref{LprB}]
	Combining \eqref{Lt} and Lemma \ref{Ltqr} we have
	\begin{align*}
		\|\nabla u\|_{L^\infty_TL^2(\mathbb{R}^3)}+\|\nabla^2 u\|_{L^2_TL^2 (\mathbb{R}^3)}
		&\leq \|\nabla u_0\|_{L^2(\mathbb{R}^3)}
		\exp\left(C\|u\|^{\frac{4}{3-2\alpha}}_{L^{\frac{4}{{3}-2\alpha}}_T\dot{B}^{-\alpha}_{\infty,\infty}(\mathbb{R}^3)}\right)\cdot\\
		&\qquad\cdot \exp\left(C\|u\|_{L^{\frac{2}{1-\alpha},r}_T\dot{B}^{-\alpha}_{\infty,\infty}(\mathbb{R}^3)}^{\frac{2}{1-\alpha}}
		(
		\log(1+\|u\|_{L_T^\infty\dot{B}^{-\alpha}_{\infty,\infty}(\mathbb{R}^3)})
		)^{1-\frac{2}{(1-\alpha)r}}
		\right)
	\end{align*}
	By Gargliardo-Nirenberg interpolation inequality and Young's inequality, we have
	\begin{align*}
		\|u\|_{L_T^{\frac{4}{{3}-2\alpha}}\dot{B}^{-\alpha}_{\infty,\infty}(\mathbb{R}^3)}
		&\lesssim \left(\int_{0}^{T} \left(\|u\|_{\dot{B}^{0}_{2,2}(\mathbb{R}^3)}^{\alpha-\frac{1}{2}}
		\|u\|_{\dot{B}^{1}_{2,2}(\mathbb{R}^3)}^{\frac{3}{2}-\alpha}\right)^{{\frac{4}{{3}-2\alpha}}}dt\right)^{\frac{{3}-2\alpha}{4}}\\
		&\lesssim \|u\|_{L_T^\infty\dot{B}^{0}_{2,2}(\mathbb{R}^3)}^{\alpha-\frac{1}{2}}
		\|u\|_{L_T^2\dot{B}^{1}_{2,2}(\mathbb{R}^3)}^{\frac{3}{2}-\alpha}\lesssim \|u\|_{L_T^\infty\dot{B}^{0}_{2,2}(\mathbb{R}^3)}
		+\|u\|_{L_T^2\dot{B}^{1}_{2,2}(\mathbb{R}^3)}.
	\end{align*}
	Multiply $ u$ on \eqref{NS} and integral over $dx$ we have $L^2$ energy inequality
	\begin{align*}
		\|u(t,\cdot)\|_{L^2(\mathbb{R}^3)}^2+2 \int_{0}^{t}\|\nabla u(s,\cdot)\|_{L^2(\mathbb{R}^3)}^2 ds\leq\|u_0\|_{L^2(\mathbb{R}^3)}^2.
	\end{align*}
	Therefore, we obtain that
	\begin{align*}
		\|u\|_{L_T^{\frac{4}{{3}-2\alpha}}\dot{B}^{-\alpha}_{\infty,\infty}(\mathbb{R}^3)}
		&\lesssim \|u_0\|_{L^2(\mathbb{R}^3)},
	\end{align*}
	which implies that  
	\begin{align}
		\exp\left(C\|u\|^{\frac{4}{3-2\alpha}}_{L_T^{\frac{4}{{3}-2\alpha}}\dot{B}^{-\alpha}_{\infty,\infty}(\mathbb{R}^3)}\right)\leq \exp\left(C\|u_0\|^{\frac{4}{3-2\alpha}}_{L^2(\mathbb{R}^3)}\right).
	\end{align}
	Note that
	\begin{align*}
		\|u\|_{L_T^\infty\dot{B}^{-\alpha}_{\infty,\infty}(\mathbb{R}^3)}\lesssim \|u\|	^{\frac{3}{2}-\alpha}_{L_T^\infty\dot{H}^{1}(\mathbb{R}^3)}\|u\|	^{\alpha-\frac{1}{2}}_{L_T^\infty{L}^{2}(\mathbb{R}^3)}.
	\end{align*}
	Thus, applying Young's inequality, we obtain
	{\small	\begin{align*}
			&\exp\left(C\|u\|_{L_T^{\frac{2}{1-\alpha},r}\dot{B}^{-\alpha}_{\infty,\infty}(\mathbb{R}^3)}^{\frac{2}{1-\alpha}}
			\left(
			\log\left(1+\|u\|_{L_T^\infty\dot{B}^{-\alpha}_{\infty,\infty}(\mathbb{R}^3)}\right)
			\right)^{1-\frac{2}{(1-\alpha)r}}
			\right)\\
			&\quad\leq\exp\left(C\|u\|_{L_T^{\frac{2}{1-\alpha},r}\dot{B}^{-\alpha}_{\infty,\infty}(\mathbb{R}^3)}^{r}\right)
			\exp
			\left(\log\left(1+\|u\|_{L_T^\infty\dot{B}^{-\alpha}_{\infty,\infty}(\mathbb{R}^3)}\right)\right)
			\\
			&\quad\leq \exp\left(C\|u\|_{L_T^{\frac{2}{1-\alpha},r}\dot{B}^{-\alpha}_{\infty,\infty}(\mathbb{R}^3)}^{r}\right)
			\left(1+\|u\|	^{\frac{3}{2}-\alpha}_{L_T^\infty\dot{H}^{1}(\mathbb{R}^3)}\|u\|	^{\alpha-\frac{1}{2}}_{L_T^\infty{L}^{2}(\mathbb{R}^3)}\right).
	\end{align*}}
	Therefore,
	\begin{align*}
		&	\|\nabla u\|_{L^\infty_TL^2(\mathbb{R}^3)}+\|\nabla^2 u\|_{L^2_TL^2(\mathbb{R}^3)}\\
		&\quad	\leq\|\nabla u_0\|_{L^2(\mathbb{R}^3)}\exp\left(C\|u_0\|^{\frac{4}{3-2\alpha}}_{L^2(\mathbb{R}^3)}\right)\exp\left(C\|u\|_{L_T^{\frac{2}{1-\alpha},r}\dot{B}^{-\alpha}_{\infty,\infty}(\mathbb{R}^3)}^{r}\right)
		\left(1+\|u\|	^{\frac{3}{2}-\alpha}_{L_T^\infty\dot{H}^{1}(\mathbb{R}^3)}\|u\|	^{\alpha-\frac{1}{2}}_{L_T^\infty{L}^{2}(\mathbb{R}^3)}\right)\\
		&\quad	\leq\|\nabla u_0\|_{L^2(\mathbb{R}^3)}\exp\left(C\|u_0\|^{\frac{4}{3-2\alpha}}_{L^2(\mathbb{R}^3)}\right)\exp\left(C\|u\|_{L_T^{\frac{2}{1-\alpha},r}\dot{B}^{-\alpha}_{\infty,\infty}(\mathbb{R}^3)}^{r}\right)\\
		&\quad\quad+C\|\nabla u_0\|^{\frac{2}{2\alpha-1}}_{L^2(\mathbb{R}^3)}\exp\left(C\|u_0\|^{\frac{4}{3-2\alpha}}_{L^2(\mathbb{R}^3)}\right)\exp\left(C\|u\|_{L_T^{\frac{2}{1-\alpha},r}\dot{B}^{-\alpha}_{\infty,\infty}(\mathbb{R}^3)}^{r}\right)\|u\|	_{L_T^\infty{L}^{2}(\mathbb{R}^3)}+\frac{1}{2}\|u\|_{L_T^\infty\dot{H}^{1}(\mathbb{R}^3)}\\
		&\quad\leq C\|\nabla u_0\|_{L^2(\mathbb{R}^3)}
		\left(1+\|\nabla u_0\|^{\frac{3-2\alpha}{2\alpha-1}}_{L^2(\mathbb{R}^3)} \|u_0\|_{L^2} \right)\exp\left(C\|u_0\|^{\frac{4}{3-2\alpha}}_{L^2(\mathbb{R}^3)}\right)\exp\left(C\|u\|_{L_T^{\frac{2}{1-\alpha},r}\dot{B}^{-\alpha}_{\infty,\infty}(\mathbb{R}^3)}^{r}\right)+\frac{1}{2}\|u\|_{L_T^\infty\dot{H}^{1}(\mathbb{R}^3)}.
	\end{align*}
	Finally, absorbing $\frac{1}{2}\|u\|_{L_T^\infty\dot{H}^{1}(\mathbb{R}^3)}$ by left hand side, we complete the proof.
\end{proof}\\

\begin{proof}[Proof of Theorem  \ref{endpoint case}]
	To obtain our results, we firstly decompose $u$ into high frequency $u_h$ and low frequency $u_l$.  Thus,
	\begin{align*}
		u(t,x)=u_h(t,x)+u_l(t,x)
	\end{align*}
	where, 
	\begin{align*}
		u_l(t,x)=\sum_{k\leq j}\dot{\Delta}_ku(t,x),\quad u_h(t,x)=\sum_{k\geq j+1}\dot{\Delta}_ku(t,x),
	\end{align*}
	and $j=j(t)$ will be determined later.\\ 
	On account of 
	\begin{align*}
		u_1\in L_{T}^{\frac{2}{1-\alpha}}
		\dot{B}^{-\alpha}_{\infty,\infty}(\mathbb{R}^3),\quad
		\|u_2\|_{L_T^\infty\dot{B}^{-1}_{\infty,\infty}(\mathbb{R}^3)} \leq \epsilon,
	\end{align*}
	we have
	\begin{align}\label{frequency u}
		\|
		\dot{\Delta}_ku_1(t)\|_{L^\infty(\mathbb{R}^3)} \leq 2^{k\alpha} U(t),
		\quad 	\|\dot{\Delta}_ku_2(t)\|_{L^\infty(\mathbb{R}^3)} \leq 2^k \epsilon,\quad \forall k\in\mathbb{Z}
	\end{align}
	here  $\|U(t)\|_ {L^{\frac{2}{1-\alpha}}([0,T])}\sim\|u_1\|_{L^{\frac{2}{1-\alpha}}_T\dot{B}^{-\alpha}_{\infty,\infty}(\mathbb{R}^3)}$. \\
	Therefore we obtain that
	\begin{align*}
		\|u_h(t)\|_{\dot{B}^{-1}_{\infty,\infty}(\mathbb{R}^3)}=&\sup_{k\in\mathbb{Z}} 2^{-k}\|
		\dot{\Delta}_ku_h(t)\|_{L^\infty(\mathbb{R}^3)}
		\lesssim \sup_{k\geq j} 2^{-k}\|
		\dot{\Delta}_ku(t)\|_{L^\infty(\mathbb{R}^3)}\\
		\leq& \sup_{k\geq j} 2^{-k}\|
		\dot{\Delta}_ku_1(t)\|_{L^\infty(\mathbb{R}^3)}+\sup_{k\geq j} 2^{-k}\|
		\dot{\Delta}_ku_2(t)\|_{L^\infty(\mathbb{R}^3)}\\
		&\leq \sup_{k\geq j} 2^{-k(1-\alpha)}U(t)+\epsilon=2^{-j(1-\alpha)}U(t)+\epsilon.
	\end{align*}
	To treat the term $u_l$, we write
	\begin{align*}
		\|u_l\|_{\dot{B}^{-\alpha}_{\infty,\infty}(\mathbb{R}^3)}=&\sup_{k\in\mathbb{Z}} 2^{-k\alpha}\|
		\dot{\Delta}_ku_l(t)\|_{L^\infty(\mathbb{R}^3)}
		\lesssim \sup_{k\leq j+1} 2^{-k\alpha}\|
		\dot{\Delta}_ku(t)\|_{L^\infty(\mathbb{R}^3)}\\
		\leq& \sup_{k\leq j+1} 2^{-k\alpha}\|
		\dot{\Delta}_ku_1(t)\|_{L^\infty(\mathbb{R}^3)}+\sup_{k\leq j+1} 2^{-k\alpha}\|
		\dot{\Delta}_ku_2(t)\|_{L^\infty(\mathbb{R}^3)}\\
		&
		\leq U(t)+\epsilon \sup_{k\leq j+1} 2^{k(1-\alpha)}\lesssim U(t)+\epsilon  2^{j(1-\alpha)}.
	\end{align*}
	Taking $j=j(t)\in \mathbb{Z}$ such that $U(t)\sim\epsilon  2^{j(1-\alpha)}$, then we deduce
	\begin{align*}
		\|u_h\|_{\dot{B}^{-1}_{\infty,\infty}(\mathbb{R}^3)}\lesssim \epsilon,\quad
		\|u_l\|_{\dot{B}^{-\alpha}_{\infty,\infty}(\mathbb{R}^3)}\lesssim U(t),
	\end{align*}
	which implies
	\begin{align*}
		\|u_h\|_{L^\infty_T\dot{B}^{-1}_{\infty,\infty}(\mathbb{R}^3)}\lesssim \epsilon, \quad
		\|u_l\|_{L_{T}^{\frac{2}{1-\alpha}}\dot{B}^{-\alpha}_{\infty,\infty}(\mathbb{R}^3)}\lesssim \|u_1\|_{L_{T}^{\frac{2}{1-\alpha}}\dot{B}^{-\alpha}_{\infty,\infty}(\mathbb{R}^3)}.
	\end{align*}
	According to \eqref{NS ineq}, \eqref{NS ineq2} and Lemma \ref{1}, we have
	\begin{align*}
		\partial_t\|\nabla u(t)\|^2_{L^2(\mathbb{R}^3)}+\|\Delta u(t)\|^2_{L^2(\mathbb{R}^3)}&\leq C \|\nabla u_h\|_{L^3(\mathbb{R}^3)}^3+C \|\nabla u_l\|_{L^3(\mathbb{R}^3)}^3\\
		&\leq C\|u_h\|_{\dot{B}^{-1}_{\infty,\infty}(\mathbb{R}^3)}\|\nabla^2 u_h\|^2_{L^2(\mathbb{R}^3)}+\frac{1}{2}\|\nabla^2 u_l\|_{L^2(\mathbb{R}^3)}^2+C\|u_l\|_{\dot{B}^{-\alpha}_{\infty,\infty}(\mathbb{R}^3)}^{\frac{2}{1-\alpha}}\|\nabla u_l\|^{2}_{L^2(\mathbb{R}^3)}\\
		&\leq \left(C\epsilon_0+\frac{1}{2}\right)\|\nabla^2 u\|^2_{L^2(\mathbb{R}^3)}+C\|u_l\|_{\dot{B}^{-\alpha}_{\infty,\infty}(\mathbb{R}^3)}^{\frac{2}{1-\alpha}}\|\nabla u\|^{2}_{L^2(\mathbb{R}^3)}.
	\end{align*}
	Taking $\epsilon_0\leq\frac{1}{4C+1}$ to absorb the first term of right hand side, we get
	\begin{align*}
		\|\nabla u\|_{L^\infty_TL^2(\mathbb{R}^3)}+\|\nabla^2 u\|_{L^2_TL^2 (\mathbb{R}^3)}
		&\leq \|\nabla u_0\|_{L^2(\mathbb{R}^3)}\exp\left(C\int_{0}^{T}\|u_l(s)\|_{\dot{B}^{-\alpha}_{\infty,\infty}(\mathbb{R}^3)}^{\frac{2}{1-\alpha}}ds\right)\\
		&\leq \|\nabla u_0\|_{L^2(\mathbb{R}^3)}\exp\left(C\int_{0}^{T}\|u_1(s)\|_{\dot{B}^{-\alpha}_{\infty,\infty}(\mathbb{R}^3)}^{\frac{2}{1-\alpha}}ds\right).
	\end{align*}
\end{proof}

\section{Proof of Theorem \ref{SQGTH}, \ref{SQGTH 1} and \ref{SQGTH 2}}
First, we 	introduce the following a bilinear operator $\mathcal{B}(g_1,g_2)$,
\begin{align}\label{B}
	\mathcal{B}(g_1,g_2)=
	\left[\Delta\left(\nabla^\perp(-\Delta)^{-\frac{1}{2}} g_1\cdot\nabla g_2\right)-\nabla^\perp(-\Delta)^{-\frac{1}{2}} g_1\cdot\nabla \Delta g_2 \right],
\end{align}
to deal with the nonlinear term of SQG equation.
\begin{lemma}\label{2}
	For any $\alpha\in(0,1), \varepsilon\in [0,\frac{\alpha}{2})$, we have
	\small{	\begin{align}\label{g1g2g3}
			\left|\int_{\mathbb{R}^2}\mathcal{B}(g_1,g_2)\Delta g_3 \right|
			&\lesssim \left(\|g_1\|_{\dot{H}^{2+\frac{\alpha}{2}}(\mathbb{R}^2)}\|g_2\|_{\dot{C}^{1-\alpha}(\mathbb{R}^2)}+\|g_2\|_{\dot{H}^{2+\frac{\alpha}{2}-\epsilon}(\mathbb{R}^2)} \|g_1\|_{\dot{C}^{1-\alpha+\epsilon}(\mathbb{R}^2)}\right)\|g_3\|_{\dot{H}^{2+\frac{\alpha}{2}}(\mathbb{R}^2)}.
			\\\label{g1g2g3'}
			\left|\int_{\mathbb{R}^2}\mathcal{B}(g_1,g_2)\Delta g_3 \right|
			&\lesssim \left(\|g_2\|_{\dot{H}^{2+\frac{\alpha}{2}}(\mathbb{R}^2)}\|g_1\|_{\dot{C}^{1-\alpha}(\mathbb{R}^2)}
			+
			\|g_1\|_{\dot{H}^{2+\frac{\alpha}{2}-\epsilon}(\mathbb{R}^2)}\|g_2\|_{\dot{C}^{1-\alpha+\epsilon}(\mathbb{R}^2)}\right)
			\|g_3\|_{\dot{H}^{2+\frac{\alpha}{2}}(\mathbb{R}^2)}.
			\\\label{ggg3}
			\left|\int_{\mathbb{R}^2}\mathcal{B}(g,g)\Delta g_3 \right|
			&\lesssim \|g\|_{\dot{H}^2(\mathbb{R}^2)}^{\frac{2\epsilon}{\alpha}}\|g\|_{\dot{H}^{2+\frac{\alpha}{2}}(\mathbb{R}^2)}^{\frac{\alpha-2\epsilon}{\alpha}}\|g\|_{\dot{C}^{1-\alpha+\epsilon}(\mathbb{R}^2)}\|g_3\|_{\dot{H}^{2+\frac{\alpha}{2}}(\mathbb{R}^2)}.
			\\\label{ggg3'}
			\left|\int_{\mathbb{R}^2}\mathcal{B}(g,g)\Delta g_3 \right|
			&\lesssim
			\|g\|_{\dot{H}^{2+\frac{\alpha}{2}}(\mathbb{R}^2)}
			\|g\|_{\dot{C}^{1-\alpha}(\mathbb{R}^2)}
			\|g_3\|_{\dot{H}^{2+\frac{\alpha}{2}}(\mathbb{R}^2)}.
	\end{align}}
\end{lemma}
\begin{proof}
	We have
	\begin{align*}
		\left|\int_{\mathbb{R}^2}	\mathcal{B}(g_1,g_2)\Delta g_3 \right|
		= 
		\left|\int_{\mathbb{R}^2} \Delta^{-\frac{\alpha}{4}}\mathcal{B}(g_1,g_2) \cdot \Delta^{1+\frac{\alpha}{4}}g_3
		\right|
		\leq
		\|\Delta^{-\frac{\alpha}{4}}\mathcal{B}(g_1,g_2)\|_{L^2(\mathbb{R}^2)}\|g_3\|_{\dot{H}^{2+\frac{\alpha}{2}}(\mathbb{R}^2)}.
	\end{align*}
	For the term $\mathcal{B}(g_1,g_2)$, note that
	\begin{align*}
		\mathcal{B}(g_1,g_2)&=\Delta\left(\nabla^\perp(-\Delta)^{-\frac{1}{2}} g_1\cdot\nabla g_2\right)-\nabla^\perp(-\Delta)^{-\frac{1}{2}} g_1\cdot\nabla \Delta g_2\\
		&=\nabla g_2 \cdot \Delta\left(\nabla^\perp(-\Delta)^{-\frac{1}{2}} g_1\right)	+2\sum_{j=1}^{2}\partial_j \left(\nabla^\perp(-\Delta)^{-\frac{1}{2}} g_1\right)\cdot\partial_j \left(\nabla g_2\right)\\
		&=\sum_{j=1}^{2}\partial_j\left(\nabla g_2 \cdot \partial_j\left(\nabla^\perp(-\Delta)^{-\frac{1}{2}} g_1\right)\right)+\sum_{j=1}^{2}\partial_j \left(\nabla^\perp(-\Delta)^{-\frac{1}{2}} g_1\right)\cdot\partial_j \left(\nabla g_2\right)\\
		&=\nabla\cdot\left[\nabla g_2 \left(\nabla\left(\nabla^\perp(-\Delta)^{-\frac{1}{2}} g_1\right)\right)\right]^\top-\nabla^\perp\cdot\left[\left(\nabla(-\Delta)^{-\frac{1}{2}} \nabla g_1\right)(\nabla g_2)^{\top}\right],
	\end{align*}
	combining Definition \ref{Bony} we obtain that
	\begin{align*}
		\partial_{k}g_2(-\Delta)^{-\frac{1}{2}}\partial_{ij}g_1
		=T_{hl}((-\Delta)^{-\frac{1}{2}}\partial_{ij}g_1,\partial_{k}g_2)
		+
		T_{lh}((-\Delta)^{-\frac{1}{2}}\partial_{ij}g_1,\partial_{k}g_2)
		+
		R((-\Delta)^{-\frac{1}{2}}\partial_{ij}g_1,\partial_{k}g_2),
	\end{align*}
	\textbf{(1)} 
	Thus, to prove the first inequality, we have
	\begin{align*}
		\|\Delta^{-\frac{\alpha}{4}}\mathcal{B}(g_1,g_2)\|_{L^2(\mathbb{R}^2)}
		&\lesssim \sum_{i,j,k,l}\left\|\Delta^{-\frac{\alpha}{4}}\partial_l\left(\partial_{k}g_2(-\Delta)^{-\frac{1}{2}}\partial_{ij}g_1\right)\right\|_{L^2(\mathbb{R}^2)}
		\lesssim \sum_{i,j,k}\left\|\partial_{k}g_2(-\Delta)^{-\frac{1}{2}}\partial_{ij}g_1\right\|_{\dot{B}^{1-\frac{\alpha}{2}}_{2,2}(\mathbb{R}^2)}\\
		&\lesssim\sum_{i,j,k}
		\left(
		\|T_{hl}((-\Delta)^{-\frac{1}{2}}\partial_{ij}g_1,\partial_{k}g_2)\|_{\dot{B}^{1-\frac{\alpha}{2}}_{2,2}(\mathbb{R}^2)}
		+ 
		\|T_{lh}((-\Delta)^{-\frac{1}{2}}\partial_{ij}g_1,\partial_{k}g_2)\|_{\dot{B}^{1-\frac{\alpha}{2}}_{2,2}(\mathbb{R}^2)}\right.\\
		&\left.\qquad\qquad+
		\|R((-\Delta)^{-\frac{1}{2}}\partial_{ij}g_1,\partial_{k}g_2)\|_{\dot{B}^{1-\frac{\alpha}{2}}_{2,2}(\mathbb{R}^2)}
		\right).
	\end{align*}
	For the first term, using Lemma \ref{T_hl}, we have		
	\begin{align*}
		\|T_{hl}((-\Delta)^{-\frac{1}{2}}\partial_{ij}g_1,\partial_{k}g_2)\|_{\dot{B}^{1-\frac{\alpha}{2}}_{2,2}(\mathbb{R}^2)}
		&\leq \|(-\Delta)^{-\frac{1}{2}}\partial_{ij}g_1\|_{\dot{B}^{1+\frac{\alpha}{2}}_{2,2}(\mathbb{R}^2)}
		\|\partial_{k}g_2\|_{\dot{B}_{\infty,\infty}^{-\alpha}(\mathbb{R}^2)}\\
		&\leq \|g_1\|_{\dot{H}^{2+\frac{\alpha}{2}}(\mathbb{R}^2)} \|g_2\|_{\dot{C}^{1-\alpha}(\mathbb{R}^2)}.
	\end{align*}	
	For the rest of the term, it is clear that	
	\begin{align*}
		&\|T_{lh}((-\Delta)^{-\frac{1}{2}}\partial_{ij}g_1,\partial_{k}g_2)\|_{\dot{B}^{1-\frac{\alpha}{2}}_{2,2}(\mathbb{R}^2)}
		+
		\|R((-\Delta)^{-\frac{1}{2}}\partial_{ij}g_1,\partial_{k}g_2)\|_{\dot{B}^{1-\frac{\alpha}{2}}_{2,2}(\mathbb{R}^2)}\\
		&\quad\leq \|\partial_{k}g_2\|_{\dot{B}^{1+\frac{\alpha}{2}-\epsilon}_{2,2}(\mathbb{R}^2)}
		\|(-\Delta)^{-\frac{1}{2}}\partial_{ij}g_1\|_{\dot{B}_{\infty,\infty}^{-\alpha+\epsilon}(\mathbb{R}^2)}\\
		&\quad \leq \|g_2\|_{\dot{H}^{2+\frac{\alpha}{2}-\epsilon}(\mathbb{R}^2)} 
		\|g_1\|_{\dot{C}^{1-\alpha+\epsilon}(\mathbb{R}^2)}
	\end{align*}
	Therefore, it is easy to check 
	\begin{align*}
		\left|\int_{\mathbb{R}^2}\mathcal{B}(g_1,g_2)\Delta g_3 \right|\lesssim \left(\|g_1\|_{\dot{H}^{2+\frac{\alpha}{2}}(\mathbb{R}^2)}\|g_2\|_{\dot{C}^{1-\alpha}(\mathbb{R}^2)}+\|g_2\|_{\dot{H}^{2+\frac{\alpha}{2}-\epsilon}(\mathbb{R}^2)} \|g_1\|_{\dot{C}^{1-\alpha+\epsilon}(\mathbb{R}^2)}\right)\|g_3\|_{\dot{H}^{2+\frac{\alpha}{2}}(\mathbb{R}^2)},
	\end{align*}
	which completes the proof of \eqref{g1g2g3}.\\
	\textbf{(2)} To prove the second one, using Lemma \ref{T_hl}, we have	
	\begin{align*}
		\|T_{lh}((-\Delta)^{-\frac{1}{2}}\partial_{ij}g_1,\partial_{k}g_2)\|_{\dot{B}^{1-\frac{\alpha}{2}}_{2,2}(\mathbb{R}^2)}
		&\leq \|\partial_{k}g_2\|_{\dot{B}^{1+\frac{\alpha}{2}}_{2,2}(\mathbb{R}^2)}
		\|(-\Delta)^{-\frac{1}{2}}\partial_{ij}g_1\|_{\dot{B}_{\infty,\infty}^{-\alpha}(\mathbb{R}^2)}\\
		& \leq \|g_2\|_{\dot{H}^{2+\frac{\alpha}{2}}} 
		\|g_1\|_{\dot{C}^{1-\alpha}}
	\end{align*}
	and
	\begin{align*}
		&\|T_{hl}((-\Delta)^{-\frac{1}{2}}\partial_{ij}g_1,\partial_{k}g_2)\|_{\dot{B}^{1-\frac{\alpha}{2}}_{2,2}(\mathbb{R}^2)}
		+
		\|R((-\Delta)^{-\frac{1}{2}}\partial_{ij}g_1,\partial_{k}g_2)\|_{\dot{B}^{1-\frac{\alpha}{2}}_{2,2}(\mathbb{R}^2)}\\
		&\quad\leq \|(-\Delta)^{-\frac{1}{2}}\partial_{ij}g_1\|_{\dot{B}^{1+\frac{\alpha}{2}-\epsilon}_{2,2}(\mathbb{R}^2)}
		\|\partial_{k}g_2\|_{\dot{B}_{\infty,\infty}^{-\alpha+\epsilon}(\mathbb{R}^2)}\\
		&\quad\leq \|g_1\|_{\dot{H}^{2+\frac{\alpha}{2}-\epsilon}(\mathbb{R}^2)} \|g_2\|_{\dot{C}^{1-\alpha+\epsilon}(\mathbb{R}^2)}.
	\end{align*}
	Therefore, it is easy to show the statement \eqref{g1g2g3'}
	\begin{align*}
		\left|\int_{\mathbb{R}^2}\mathcal{B}(g_1,g_2)\Delta g_3 \right|
		\lesssim \left(\|g_2\|_{\dot{H}^{2+\frac{\alpha}{2}}(\mathbb{R}^2)}\|g_1\|_{\dot{C}^{1-\alpha}(\mathbb{R}^2)}
		+
		\|g_1\|_{\dot{H}^{2+\frac{\alpha}{2}-\epsilon}(\mathbb{R}^2)}\|g_2\|_{\dot{C}^{1-\alpha+\epsilon}(\mathbb{R}^2)}\right)
		\|g_3\|_{\dot{H}^{2+\frac{\alpha}{2}}(\mathbb{R}^2)}.
	\end{align*}
	\textbf{(3)}
	In particular, when $g_1=g_2=g$, we can obtain
	\begin{align*}
		&\|\Delta^{-\frac{\alpha}{4}}\mathcal{B}(g,g)\|_{L^2(\mathbb{R}^2)}\\
		&\quad\lesssim \sum_{i,j,k}\left(\left\|\partial_{k}g\right\|_{\dot{B}_{2,2}^{1+\frac{\alpha}{2}-\epsilon}(\mathbb{R}^2)}\|(-\Delta)^{-\frac{1}{2}}\partial_{ij}g\|_{\dot{B}_{\infty,\infty}^{\epsilon-\alpha}(\mathbb{R}^2)}
		+\left\|(-\Delta)^{-\frac{1}{2}}\partial_{ij}g\right\|_{\dot{B}_{2,2}^{1+\frac{\alpha}{2}-\epsilon}(\mathbb{R}^2)}\|\partial_{k}g\|_{\dot{B}_{\infty,\infty}^{\epsilon-\alpha}(\mathbb{R}^2)}\right)\\
		&\quad\lesssim \|g\|_{\dot{H}^{2+\frac{\alpha}{2}-\epsilon}(\mathbb{R}^2)}\|g\|_{\dot{C}^{1-\alpha+\epsilon}(\mathbb{R}^2)}\lesssim \|g\|_{\dot{H}^2(\mathbb{R}^2)}^{\frac{2\epsilon}{\alpha}}
		\|g\|_{\dot{H}^{2+\frac{\alpha}{2}}(\mathbb{R}^2)}^{\frac{\alpha-2\epsilon}{\alpha}}
		\|g\|_{\dot{C}^{1-\alpha+\epsilon}(\mathbb{R}^2)}.
	\end{align*}
	Therefore, we get the estimate in \eqref{ggg3}
	\begin{align*}
		\left|\int_{\mathbb{R}^2}\mathcal{B}(g,g)\Delta g_3 \right|\lesssim \|g\|_{\dot{H}^2(\mathbb{R}^2)}^{\frac{2\epsilon}{\alpha}}\|g\|_{\dot{H}^{2+\frac{\alpha}{2}}}^{\frac{\alpha-2\epsilon}{\alpha}}\|g\|_{\dot{C}^{1-\alpha+\epsilon}}\|g_3\|_{\dot{H}^{2+\frac{\alpha}{2}}}.
	\end{align*}
	\textbf{(4)} Moreover, when $\epsilon=0$, we have
	\begin{align*}
		\left|\int_{\mathbb{R}^2}\mathcal{B}(g,g)\Delta g_3 \right|
		\lesssim
		\|g\|_{\dot{H}^{2+\frac{\alpha}{2}}}
		\|g\|_{\dot{C}^{1-\alpha}}
		\|g_3\|_{\dot{H}^{2+\frac{\alpha}{2}}}.
	\end{align*}
\end{proof}

\begin{proof}[Proof of Theorem \ref{SQGTH}]
	Multiply $\Delta^2 \theta$ on \eqref{SQG} and integral over $dx$ we have
	\begin{align*}
		\frac{1}{2}\partial_t\|\theta\|_{\dot{H}^2}^2+\|\theta\|^2_{\dot{H}^{2+\frac{\alpha}{2}}}=-\int_{\mathbb{R}^2}	\mathcal{B}\Delta \theta dx-\int_{\mathbb{R}^2} \left(\nabla^\perp(-\Delta)^{-\frac{1}{2}} \theta\cdot\nabla \Delta \theta\right)\Delta \theta dx,
	\end{align*}
	here we recall   the definition of $\mathcal{B}$ in \eqref{B}.\\
	Note that $\nabla\cdot\nabla^\perp=0$, thus
	\begin{align*}
		\int_{\mathbb{R}^2} \nabla^\perp(-\Delta)^{-\frac{1}{2}} \theta\cdot(\nabla \Delta \theta)\Delta \theta dx=&-	\int_{\mathbb{R}^2} \nabla\cdot\left(\nabla^\perp(-\Delta)^{-\frac{1}{2}} \theta \right)\left(\Delta \theta\right)^2 dx-	\int_{\mathbb{R}^2} \Delta \theta\left(\nabla^\perp(-\Delta)^{-\frac{1}{2}} \theta \right)\cdot\nabla \Delta \theta dx\\
		=&-		\int_{\mathbb{R}^2} \left(\nabla^\perp(-\Delta)^{-\frac{1}{2}} \theta \right)\cdot(\nabla \Delta \theta)\Delta \theta dx.
	\end{align*}
	Hence we obtain that
	\begin{align*}
		\int_{\mathbb{R}^2} \nabla^\perp(-\Delta)^{-\frac{1}{2}} \theta\cdot(\nabla \Delta \theta)\Delta \theta dx=0.
	\end{align*}
	Applying Young's inequality and the result in Lemma \ref{2}, we have
	\begin{align*}
		\frac{1}{2}\partial_t\|\theta\|_{\dot{H}^2(\mathbb{R}^2)}^2+\|\theta\|^2_{\dot{H}^{2+\frac{\alpha}{2}}(\mathbb{R}^2)}&\leq	\left|\int_{\mathbb{R}^2}\left[\Delta\left(\nabla^\perp(-\Delta)^{-\frac{1}{2}} \theta\cdot\nabla \theta\right)-\nabla^\perp(-\Delta)^{-\frac{1}{2}} \theta\cdot\nabla \Delta \theta \right]\Delta \theta \right| \\
		&\leq C	\|\theta\|_{\dot{H}^2(\mathbb{R}^2)}^{\frac{2\epsilon}{\alpha}}\|\theta\|_{\dot{H}^{2+\frac{\alpha}{2}}(\mathbb{R}^2)}^{\frac{2\alpha-2\epsilon}{\alpha}}\|\theta\|_{\dot{C}^{1-\alpha+\varepsilon}(\mathbb{R}^2)}\\
		&\leq\frac{1}{2}
		\|\theta\|_{\dot{H}^{2+\frac{\alpha}{2}}(\mathbb{R}^2)}^2
		+
		C\|\theta\|_{\dot{H}^2(\mathbb{R}^2)}^2\|\theta\|_{\dot{C}^{1-\alpha+\varepsilon}(\mathbb{R}^2)}^{\frac{\alpha}{\epsilon}}.
	\end{align*}
	Absorbing  $\frac{1}{2}\|\theta\|_{\dot{H}^{2+\frac{\alpha}{2}}(\mathbb{R}^2)}^2$ by the left hand side and combining Gr\"{o}nwall inequality we obtain that 
	\begin{align*}
		\| \theta\|_{L^\infty_T\dot{H}^2(\mathbb{R}^2)}+\|\theta\|_{L^2_T \dot{H}^{2+\frac{\alpha}{2}}(\mathbb{R}^2)}\leq \| \theta_0\|_{\dot{H}^2(\mathbb{R}^2)}\exp\left(C\int_{0}^{T}\|\theta(s)\|_{\dot{C}^{1-\alpha+\epsilon}(\mathbb{R}^2)}^{\frac{\alpha}{\epsilon}}ds\right),
	\end{align*}
	where $C$ is a constant depending on $\alpha$ and $\varepsilon$.
\end{proof}

\begin{proof}[Proof of Theorem  \ref{SQGTH 1}]
	Combining \eqref{ll} and Lemma \ref{Ltqr}, we have
	\begin{align}\label{kkk}
		&\| \theta\|_{L^\infty_T\dot{H}^2(\mathbb{R}^2)}+\|\theta\|_{L^2_T \dot{H}^{2+\frac{\alpha}{2}}(\mathbb{R}^2)}\\\nonumber
		&\quad\leq C
		\|\theta_0\|_{\dot{H}^2(\mathbb{R}^2)} \exp\left(
		\log \left(
		1+\|\theta\|_{L_T^\infty \dot{C}^{1-\alpha+\epsilon}(\mathbb{R}^2)}
		\|\theta\|_{L_T^{\frac{4+\alpha}{2-\alpha+\epsilon}} \dot{C}^{1-\alpha+\epsilon}(\mathbb{R}^2)}^{\frac{(4+\alpha)\epsilon}{2\alpha-\alpha^2-4\epsilon}}
		\right)^{1-\frac{\alpha}{\epsilon r}} \|\theta\|_{L^{\frac{\alpha}{\epsilon},r} \dot{C}^{1-\alpha+\epsilon}(\mathbb{R}^2)}^{\frac{\alpha}{\epsilon}}
		\right).
	\end{align}
	Applying Young's inequality, we obtain
	\begin{align*}
		&	\| \theta\|_{L^\infty_T\dot{H}^2(\mathbb{R}^2)}+\|\theta\|_{L^2_T \dot{H}^{2+\frac{\alpha}{2}}(\mathbb{R}^2)}\\
		&\qquad\lesssim \|\theta_0\|_{\dot{H}^2(\mathbb{R}^2)} \left(1+\|\theta\|_{L_T^\infty \dot{C}^{1-\alpha+\epsilon}(\mathbb{R}^2)} \|\theta\|_{L_T^{\frac{4+\alpha}{2-\alpha+\epsilon}} \dot{C}^{1-\alpha+\epsilon}(\mathbb{R}^2)}^{\frac{(4+\alpha)\epsilon}{2\alpha-\alpha^2-4\epsilon}}\right)
		\exp\left(\|\theta\|_{L_T^{\frac{\alpha}{\epsilon},r}\dot{C}^{1-\alpha+\epsilon}(\mathbb{R}^2)}^{r}\right).
	\end{align*}
	Applying Gagliardo-Nirenberg interpolation, we have
	\begin{align*}
		\|\theta\|_{L_T^\infty \dot{C}^{1-\alpha+\epsilon}(\mathbb{R}^2)}
		&\leq \|\theta\|_{L_T^\infty L^2(\mathbb{R}^2)}^{\frac{\alpha-\epsilon}{2}} \|\theta\|_{L_T^\infty \dot{H}^2(\mathbb{R}^2)}^{\frac{2-\alpha+\epsilon}{2}},
		\\
		\|\theta\|_{L_T^{\frac{4+\alpha}{2-\alpha+\epsilon}} \dot{C}^{1-\alpha+\epsilon}(\mathbb{R}^2)}
		&\leq \|\theta\|_{L_T^\infty L^2(\mathbb{R}^2)}^{\frac{3\alpha-2\epsilon}{4+\alpha}} \|\theta\|_{L_T^2 \dot{H}^{2+\frac{\alpha}{2}}(\mathbb{R}^2)}^{\frac{4-2\alpha+2\epsilon}{4+\alpha}}.
	\end{align*}
	Combining with Young's inequality, the following inequality holds true
	\begin{align*}
		&\|\theta_0\|_{\dot{H}^2}
		\|\theta\|_{L_T^\infty \dot{C}^{1-\alpha+\epsilon}}
		\|\theta\|_{L_T^{\frac{4+\alpha}{2-\alpha+\epsilon}} \dot{C}^{1-\alpha+\epsilon}}^{\frac{(4+\alpha)\epsilon}{2\alpha-\alpha^2-4\epsilon}}
		\exp\left(\|\theta\|_{L_T^{\frac{\alpha}{\epsilon},r}\dot{C}^{1-\alpha+\epsilon}}^{r}\right)\\
		&\leq 
		\|\theta_0\|_{\dot{H}^2}
		\|\theta\|_{L_T^\infty L^2}^{\frac{\alpha(2\alpha-\alpha^2+\alpha\epsilon)}{2(2\alpha-\alpha^2-4\epsilon)}} \|\theta\|_{L_T^\infty \dot{H}^2}^{\frac{2-\alpha+\epsilon}{2}}
		\|\theta\|_{L_T^2 \dot{H}^{2+\frac{\alpha}{2}}}^{\frac{2\epsilon(2-\alpha+\epsilon)}{2\alpha-\alpha^2-4\epsilon}}
		\exp\left(\|\theta\|_{L_T^{\frac{\alpha}{\epsilon},r}\dot{C}^{1-\alpha+\epsilon}}^{r}\right)\\
		&\leq \frac{1}{2} \|\theta\|_{L_T^\infty \dot{H}^2}
		+ \frac{1}{2} \|\theta\|_{L_T^2 \dot{H}^{2+\frac{\alpha}{2}}}
		+ C\|\theta\|_{L_T^\infty L^2}^{1+\frac{2\epsilon(\alpha-4)}{(2\alpha^2-\alpha^3+\alpha^2\epsilon)-2\epsilon(\alpha-4)}}
		\left(\|\theta_0\|_{\dot{H}^2}
		\exp\left(\|\theta\|_{L_T^{\frac{\alpha}{\epsilon},r}\dot{C}^{1-\alpha+\epsilon}}^{r}\right)\right)^{\frac{2(2\alpha-\alpha^2-4\epsilon)}{2\alpha^2-\alpha^3-2\alpha\epsilon+\alpha ^2\epsilon-8\epsilon}},
	\end{align*}
	Multiply $ \theta$ on \eqref{SQG} and integral over $dx$ we have $L^2$ energy inequality
	\begin{align*}
		\|\theta(t,\cdot)\|_{ L^2(\mathbb{R}^2)}^2 +2\int_{0}^{t} \|\theta(s,\cdot)\|_{ \dot{H}^{\frac{\alpha}{2}}(\mathbb{R}^2)}^2ds
		\leq \|\theta_0\|_{L^2(\mathbb{R}^2)}^2.
	\end{align*}
	Thus we have
	\begin{align*}
		&\|\theta_0\|_{\dot{H}^2}
		\|\theta\|_{L_T^\infty \dot{C}^{1-\alpha+\epsilon}}
		\|\theta\|_{L_T^{\frac{4+\alpha}{2-\alpha+\epsilon}} \dot{C}^{1-\alpha+\epsilon}}^{\frac{(4+\alpha)\epsilon}{2\alpha-\alpha^2-4\epsilon}}
		\exp\left(\|\theta\|_{L_T^{\frac{\alpha}{\epsilon},r}\dot{C}^{1-\alpha+\epsilon}}^{r}\right)\\
		&\leq \frac{1}{2} \|\theta\|_{L_T^\infty \dot{H}^2}
		+ \frac{1}{2} \|\theta\|_{L_T^2 \dot{H}^{2+\frac{\alpha}{2}}}
		+C \|\theta_0\|_{ L^2}^{1+\frac{2\epsilon(\alpha-4)}{(2\alpha^2-\alpha^3+\alpha^2\epsilon)-2\epsilon(\alpha-4)}}
		\|\theta_0\|_{\dot{H}^2}^{\frac{2(2\alpha-\alpha^2-4\epsilon)}{2\alpha^2-\alpha^3-2\alpha\epsilon+\alpha ^2\epsilon-8\epsilon}}
		\exp\left(C\|\theta\|_{L_T^{\frac{\alpha}{\epsilon},r}\dot{C}^{1-\alpha+\epsilon}}^{r}\right).
	\end{align*}
	Finally, absorbing $\frac{1}{2} \|\theta\|_{L_T^\infty \dot{H}^2}+\frac{1}{2} \|\theta\|_{L^2_T \dot{H}^{2+\frac{\alpha}{2}}}
	$ by left hand side of \eqref{kkk}, we have
	\begin{align*}
		&\| \theta\|_{L^\infty_T\dot{H}^2(\mathbb{R}^2)}+\|\theta\|_{L^2_T \dot{H}^{2+\frac{\alpha}{2}}(\mathbb{R}^2)}\\
		&\quad\lesssim 
		\|\theta_0\|_{\dot{H}^2}
		\exp\left(\|\theta\|_{L_T^{\frac{\alpha}{\epsilon},r}\dot{C}^{1-\alpha+\epsilon}}^{r}\right)
		+C \|\theta_0\|_{ L^2}^{1+\frac{2\epsilon(\alpha-4)}{(2\alpha^2-\alpha^3+\alpha^2\epsilon)-2\epsilon(\alpha-4)}}
		\|\theta_0\|_{\dot{H}^2}^{\frac{2(2\alpha-\alpha^2-4\epsilon)}{2\alpha^2-\alpha^3-2\alpha\epsilon+\alpha ^2\epsilon-8\epsilon}}
		\exp\left(C\|\theta\|_{L_T^{\frac{\alpha}{\epsilon},r}\dot{C}^{1-\alpha+\epsilon}}^{r}\right)\\
		&\quad\lesssim 
		\|\theta_0\|_{\dot{H}^2} 
		\left(1+ \|\theta_0\|_{L^2}
		\left(1+\|\theta_0\|_{L^2}^{-\frac{8\epsilon}{\alpha^2}}\right)
		\left(1+\|\theta_0\|_{\dot{H^2}}^{\frac{4\alpha}{\alpha^2-8\epsilon}}\right)
		\right)
		\exp\left(C\|\theta\|_{L_T^{\frac{\alpha}{\epsilon},r}\dot{C}^{1-\alpha+\epsilon}}^{r}\right),
	\end{align*}
	which we complete the proof.
\end{proof}\vspace{0.2cm}\\

\begin{proof}[Proof of Theorem  \ref{SQGTH 2}]
	We firstly decompose $\theta$ into high frequency $\theta_h$ and low frequency $\theta_l$. Thus,
	\begin{align*}
		\theta(t,x)=\theta_h(t,x)+\theta_l(t,x)
	\end{align*}
	where, 
	\begin{align*}
		\theta_l(t,x)=
		\sum_{k\leq j}\dot{\Delta}_k\theta(t,x),\quad
		\theta_h(t,x)=
		\sum_{k\geq j+1}\dot{\Delta}_k\theta(t,x),
	\end{align*}
	and $j=j(t)$ will be determined later.\\ 
	On account of 
	\begin{align*}
		\theta_1\in L_{T}^{\frac{\alpha}{\epsilon}}
		\dot{C}^{1-\alpha+\epsilon},\quad
		\|\theta_2\|_{L^\infty_T\dot{C}^{1-\alpha}} \leq \delta,
	\end{align*}
	we have
	\begin{align}\label{frequency theta}
		\|
		\dot{\Delta}_j\theta_1\|_{L^\infty(\mathbb{R}^2)} \lesssim 2^{(\alpha-1-\epsilon)j} U(t),
		\quad 	\|\dot{\Delta}_j\theta_2\|_{L^\infty(\mathbb{R}^2)} \lesssim \delta 2^{(\alpha-1)j} ,
	\end{align}
	here  $\|U(t)\|_ {L_T^{\frac{\alpha}{\epsilon}}(\mathbb{R}^2)}\sim\|\theta_1\|_{L_T^{\frac{\alpha}{\epsilon}}\dot{C}^{1-\alpha+\epsilon}(\mathbb{R}^2)}$. \\
	Therefore we obtain that
	\begin{align*}
		\|\theta_l\|_{\dot{C}^{1-\alpha+\epsilon}(\mathbb{R}^2)}
		\leq& \sup_{k\leq j} 2^{k(1-\alpha+\epsilon)}\|
		\dot{\Delta}_k\theta_l\|_{L^\infty(\mathbb{R}^2)}\\
		\leq& \sup_{k\leq j} 2^{k(1-\alpha+\epsilon)}
		\left(\|
		\dot{\Delta}_k\theta_1(t)\|_{L^\infty(\mathbb{R}^2)}+\|
		\dot{\Delta}_k\theta_2(t)\|_{L^\infty(\mathbb{R}^2)}\right)\\
		\leq& U(t)+\delta2^{\epsilon j}.
	\end{align*}
	To treat the term $\theta_h$, we have
	\begin{align*}
		\|\theta_h\|_{\dot{C}^{1-\alpha}(\mathbb{R}^2)}
		\leq& \sup_{k\geq j} 2^{k(1-\alpha)}\|
		\dot{\Delta}_k\theta\|_{L^\infty(\mathbb{R}^2)}
		\leq \delta +2^{-\epsilon j}U(t).
	\end{align*}
	Taking $j=j_t$ such that $2^{-\epsilon j_t}U(t)\sim\delta  $, then we deduce
	\begin{align*}
		\|\theta_l\|_{L_T^{\frac{\alpha}{\epsilon}}\dot{C}^{1-\alpha+\epsilon}(\mathbb{R}^2)}\lesssim \|\theta_1\|_{L_T^{\frac{\alpha}{\epsilon}}\dot{C}^{1-\alpha+\epsilon}(\mathbb{R}^2)},\quad
		\|\theta_h\|_{L_T^{\infty}\dot{C}^{1-\alpha}(\mathbb{R}^2)}\lesssim \delta.
	\end{align*}
	Multiply $\Delta \theta$ on \eqref{SQG} and integral over $dx$ we have
	\begin{align*}
		\frac{1}{2}\partial_t\|\theta\|^2_{\dot{H}^2(\mathbb{R}^2)}&+\|\theta\|^2_{\dot{H}^{2+\frac{\alpha}{2}}(\mathbb{R}^2)}
		\leq \left|\int_{\mathbb{R}^2}	\mathcal{B}\Delta \theta \right|\\
		&
		\leq\left|\int_{\mathbb{R}^2} B(\theta_l,\theta_l)\Delta\theta\right|+\left|\int_{\mathbb{R}^2} B(\theta_h,\theta_h) \Delta\theta\right|
		+\left|\int_{\mathbb{R}^2} B(\theta_l,\theta_h) \Delta\theta\right|+\left|\int_{\mathbb{R}^2} B(\theta_h,\theta_l) \Delta\theta\right|.
	\end{align*}
	Note that
	using \eqref{ggg3}, we have
	\begin{align*}
		\left|\int_{\mathbb{R}^2}\mathcal{B}(\theta_l,\theta_l)\Delta \theta \right|\lesssim \|\theta\|_{\dot{H}^2(\mathbb{R}^2)}^{\frac{2\epsilon}{\alpha}}\|\theta\|_{\dot{H}^{2+\frac{\alpha}{2}}(\mathbb{R}^2)}^{\frac{2\alpha-2\epsilon}{\alpha}}\|\theta_l\|_{\dot{C}^{1-\alpha+\epsilon}(\mathbb{R}^2)}.
	\end{align*}
	On account of \eqref{ggg3'}, 
	\begin{align*}
		\left|\int_{\mathbb{R}^2}\mathcal{B}(\theta_h,\theta_h)\Delta \theta \right|
		&\lesssim
		\|\theta_h\|_{\dot{H}^{2+\frac{\alpha}{2}}(\mathbb{R}^2)}
		\|\theta_h\|_{\dot{C}^{1-\alpha}(\mathbb{R}^2)}
		\|\theta\|_{\dot{H}^{2+\frac{\alpha}{2}}(\mathbb{R}^2)}\\
		&\leq
		\|\theta_h\|_{\dot{C}^{1-\alpha}(\mathbb{R}^2)}
		\|\theta\|_{\dot{H}^{2+\frac{\alpha}{2}}(\mathbb{R}^2)}^2.
	\end{align*}
	Moreover, we can get directly by using \eqref{g1g2g3},
	\begin{align*}
		\left|\int_{\mathbb{R}^2}\mathcal{B}(\theta_l,\theta_h)\Delta \theta \right|\lesssim \left(\|\theta_l\|_{\dot{H}^{2+\frac{\alpha}{2}}(\mathbb{R}^2)}\|\theta_h\|_{\dot{C}^{1-\alpha}(\mathbb{R}^2)}+\|\theta_h\|_{\dot{H}^{2+\frac{\alpha}{2}-\epsilon}(\mathbb{R}^2)} \|\theta_l\|_{\dot{C}^{1-\alpha+\epsilon}(\mathbb{R}^2)}\right)\|\theta\|_{\dot{H}^{2+\frac{\alpha}{2}}(\mathbb{R}^2)}.
	\end{align*}	
	Similarly, since \eqref{g1g2g3'}, it's obvious 
	\begin{align*}
		\left|\int_{\mathbb{R}^2}\mathcal{B}(\theta_h,\theta_l)\Delta \theta \right|
		\lesssim \left(\|\theta_l\|_{\dot{H}^{2+\frac{\alpha}{2}}(\mathbb{R}^2)}\|\theta_h\|_{\dot{C}^{1-\alpha}(\mathbb{R}^2)}
		+
		\|\theta_h\|_{\dot{H}^{2+\frac{\alpha}{2}-\epsilon}(\mathbb{R}^2)}\|\theta_l\|_{\dot{C}^{1-\alpha+\epsilon}(\mathbb{R}^2)}\right)
		\|\theta\|_{\dot{H}^{2+\frac{\alpha}{2}}(\mathbb{R}^2)}.
	\end{align*}
	Thus, it's easy to check
	\begin{align*}
		\partial_t\|\theta\|^2_{\dot{H}^2(\mathbb{R}^2)}+\|\theta\|^2_{\dot{H}^{2+\frac{\alpha}{2}}(\mathbb{R}^2)}
		&\leq C 	\|\theta_l\|_{\dot{C}^{1-\alpha+\epsilon}(\mathbb{R}^2)}
		\|\theta\|^{\frac{2\alpha-2\epsilon}{\alpha}}_{\dot{H}^{2+\frac{\alpha}{2}}(\mathbb{R}^2)}\|\theta\|^{\frac{2\epsilon}{\alpha}}_{\dot{H}^{2}(\mathbb{R}^2)}
		+
		\|\theta_h\|_{\dot{C}^{1-\alpha}(\mathbb{R}^2)}
		\|\theta\|^{2}_{\dot{H}^{2+\frac{\alpha}{2}}(\mathbb{R}^2)}\\
		&\leq (C\delta+\frac{1}{2})\|\theta\|^2_{\dot{H}^{2+\frac{\alpha}{2}}(\mathbb{R}^2)}
		+
		C\|\theta\|^2_{\dot{H}^{2}(\mathbb{R}^2)}\|\theta\|_{C^{1-\alpha+\epsilon}(\mathbb{R}^2)}^{\frac{\alpha}{\epsilon}}.
	\end{align*}
	Taking $\delta \leq \delta_0=\frac{1}{4C+1}$, absorbing $(C\delta+\frac{1}{2})\|\theta\|^2_{\dot{H}^{2+\frac{\alpha}{2}}(\mathbb{R}^2)}$. By Gr\"{o}nwall inequality, we complete the proof
	{	\begin{align*}
			\|\theta\|_{L^\infty_T\dot{H}^2(\mathbb{R}^2)}+\|\theta\|_{L_T^2\dot{H}^{2+\frac{\alpha}{2}}(\mathbb{R}^2) }
			&\leq \|\theta_0\|_{\dot{H}^2(\mathbb{R}^2)}\exp\left(C\int_{0}^{T}\|\theta_1(s)\|_{C^{1-\alpha+\epsilon}(\mathbb{R}^2)}^{\frac{\alpha}{\epsilon}}dt\right).
	\end{align*}}
\end{proof}\vspace{0.2cm}\\
\textbf{Acknowledgements:} This research is funded by Vietnam National University Ho Chi Minh City (VNU-HCM) under grant number T2022-18-01. The authors are grateful to Professor Quoc Hung Nguyen, who introduced this project to us and patiently guided, supported, and encouraged us during this work.

\end{document}